\documentclass[envcountsame,envcountsect,runningheads,a4paper]{llncs}
\usepackage{graphicx,version,stmaryrd,amssymb,amsmath,url}

\usepackage{latexsym}
\usepackage{amsmath,amssymb,amsthm}
\usepackage{stmaryrd}
\usepackage{atbeginend}
\usepackage{enumitem}

\newcommand{\card}[1]{|#1|}

\newcommand{\fin}[1]{[#1]}

\newcommand{\mbb}{\mathbb}

\newcommand{\zz}{\mbb{Z}}

\renewcommand{\emptyset}{\varnothing}

\newcommand{\lt}{<}

\newcommand{\where}{\mid}

\newcommand{\mfrac}{\frac}
\renewcommand{\tfrac}[2]{#1 \mdiv #2}

\newcommand{\proves}{\vdash}

\newcommand{\cterm}{\underline}

\newcommand{\defd}{\mathrel{:=}}

\newcommand{\mdiv}{\mathbin{/}}

\newcommand{\Md}{\mathrm{Md}}

\newcommand{\sigdmd}{\rSigma_{\Md}^\mathrm{d}}

\newcommand{\eqnsdmd}{\rE_{\Md}^\mathrm{d}}

\newcommand{\Rat}{\mbb{Q}}

\newcommand{\Ratzd}{\Rat_0^\mathrm{d}}

\newcommand{\Compl}{\mbb{C}}

\newcommand{\Complzd}{\Compl_{\hsp{.1}0}^\mathrm{d}}

\newcommand{\set}[1]{\{ #1 \}}

\newcommand{\Union}[1]{\bigcup_{#1}}

\newcommand{\Natpos}{\mathbb{N}^+}

\newcommand{\displstretch}{\renewcommand{\arraystretch}{1.15}}

\newcommand{\tblstretch}{\renewcommand{\arraystretch}{1.25}}
\newcommand{\eqntblsize}{\small}

\newcommand{\rE}{\mathit{E}}
\newcommand{\rSigma}{\mathit{\Sigma}}

\newcommand{\hsp}[1]{\hspace*{#1em}}

\title{Universality of Univariate Mixed Fractions in Divisive Meadows}

\author{%
  Jan A.~Bergstra \inst{1}
  \and 
  Inge Bethke \inst{1}
  \and 
  Dimitri Hendriks \inst{2}
}

\institute{
  Informatics Institute, Faculty of Science, University of Amsterdam \\
  Science Park~904, 1098~XH Amsterdam, The Netherlands \\
  \email{\{j.a.bergstra, i.bethke\}@uva.nl}
  \and
  Minstroom Research bv, Utrecht, The Netherlands \\
  \email{dimitrihendriks@gmail.com}
}

\begin{document}

\maketitle

\begin{abstract}
Univariate fractions can be transformed to mixed fractions in the equational theory of meadows of characteristic zero.
\end{abstract}

\section{Introduction}\label{sec:intro}
\emph{Meadows} are expansions of fields, and more generally of Von Neumann regular rings, 
with an inverse or division operation made total by having $0^{-1} = 0$ or $x/0=0$. 
If only an inverse is present we speak of an \emph{inversive} meadow 
and if only a division is present we speak of a \emph{divisive} meadow.
For axioms of meadows taking care of both cases we refer to~\cite{BergstraM2011a}. 
Meadows were introduced in~\cite{BergstraT2007} as a tool in an effort to develop concise algebraic specifications 
of the rational numbers. 
In the line of that work it was recently shown in~\cite{BB2016} 
that finite algebraic extensions of the rationals admit a finite initial algebra specification, 
while leaving open the question if the adjunction of a transcendental element to the meadow of 
rational numbers leads to a structure that has a finite algebraic specification.

Having equational specifications of meadows available it is plausible to investigate 
the equational theory of fractions, and initial work on that topic 
has been carried out in~\cite{BergstraM2016}. 
Here we will consider fractions as being expressions for divisive meadows 
with division as a leading function symbol. 
As it turns out the equational theory of fractions (as terms) in meadows 
presents a range of interesting problems. 

In~\cite{BergstraM2016} the question is considered whether (closed or open) fractions 
written in the signature of divisive meadows, can be rewritten into simple fractions, 
and which particular meadows admit such a transformation. 
We will need some terminology on fractions.
\emph{Fractions} are terms over $\sigdmd=\{0,1,-,+,\cdot ,/, \}$, the signature of divisive meadows,
whose head symbol is division $\mdiv$. 
As every term $t$ is provably equal to the fraction $t/1$,
fractions are universal over all terms modulo provable equality.
A \emph{closed} fraction is a fraction without variables. 
A \emph{univariate} fraction has at most one free variable.
\emph{Simple} fractions are fractions that have no fractions as proper subterms.
A \emph{polynomial} is an expression 
built from variables and closed simple fractions, and symbols for addition, multiplication and exponentiation.
In a standardised form, a polynomial in one variable~$x$ 
is of the form
$\underline{a_n} \cdot x^n + \underline{a_{n-1}} \cdot x^{n-1} + \cdots + \underline{a_1} \cdot x + \underline{a_0}$, 
where $\underline{a_i}$ is an expression for the rational number $a_i \in \Rat$ ($0 \le i \le n$). 
A \emph{mixed} fraction is a sum of a polynomial and a simple fraction. 
In~\cite{BergstraM2016} it is shown that in the setting of meadows all fractions can be rewritten 
into sums of simple fractions. 
In addition it is shown that an arbitrary number of summands may be needed for this rewrite. 
The latter proof uses the presence of an arbitrary number of variables. 
In this paper we will consider the special case of univariate fractions, 
fractions involving a single variable at most. 
We will find that in the univariate case and working in characteristic zero,
mixed fractions suffice to express all univariate fractions. This is Theorem~\ref{thm:one:var}. 
We will refer to this fact as the universality of univariate mixed fractions for univariate fractions.


Strengthening this result, we prove in Section 4 that also in the meadow of complex numbers 
univariate mixed fractions are universal. This is Theorem~\ref{thm:one:varb}.
This result has a proof theoretic consequence|Theorem~\ref{thm:one:varc}|which covers all meadows of characteristic zero. We end the section with an application of Theorem~\ref{thm:one:varb} yielding a decidability result on finite support summation as introduced in \cite{Bergstra2016}.

\section{Preliminaries}
$\eqnsdmd$ is the set of axioms of a commutative ring with a multiplicative identity element augmented by three axioms for division given in Table \ref{axioms}.
\begin{table}[t]
\centering
\hrule
\begin{align*}
	(x+y)+z &= x + (y + z)\\
	x+y     &= y+x\\
	x+0     &= x\\
	x+(-x)  &= 0\\
	(x \cdot y) \cdot  z &= x \cdot  (y \cdot  z)\\
	x \cdot  y &= y \cdot  x\\
	1\cdot x &= x \\
	x \cdot  (y + z) &= x \cdot  y + x \cdot  z\\
	1/(1/x)&=x \\
	 (x\cdot x)/x&= x\\
	 x/y&=x\cdot(1/y)
\end{align*}
\hrule
\caption{The set $\eqnsdmd$ of axioms for devisive meadows}
\label{axioms}
\end{table}
For a term~$t$ over the signature $\sigdmd$ having one or more occurrences of precisely one variable~$x$ we sometimes write $t(x)$, 
either to explicitly indicate the occurrences of $x$ in the expression~$t$,
or, in the model, to explicitly write $t$ as a function.
We then write $t(a)$ to denote the term or value obtained 
by uniformly substituting $a$ for $x$ in $t$. 
We use $\Ratzd$ and $\Compl^d_0$ to denote the meadows of 
rational and complex numbers, respectively, with divisive notation, 
and use the same notation for their carrier sets. 
A \emph{simple fraction} is an expression of the form 
$\tfrac{p}{q}$ where $p$ and $q$ are terms without division symbol~$\mdiv$. 
Every expression without division symbol (thus involving only integer coefficients)
is provably equal (in meadows) to a polynomial in standardised form.
Thus each simple fraction can be written as $\tfrac{r}{s}$ with $r$ and $s$
polynomials with integer coefficients. 
As usual, a term $p \mdiv q$ is sometimes written $\frac{p}{q}$.
For an integer $n > 0$ we write $\fin{n}$ to denote the set $\set{1,2,\ldots,n}$.
For a rational number~$a$ we will write $\underline{a}$ to denote a term with that value in $\Ratzd$.
We take the liberty to omit underlining if no confusion is likely to arise.

\newcommand{\localcite}{\cite[Thm.~21]{BergstraM2016}}
\begin{lemma}[\localcite]\label{lem:cees}
  For each term $t$ over the signature $\sigdmd$, 
  there is a finite number of simple fractions $f_1, \ldots, f_n$
  such that $\eqnsdmd \proves t = f_1 + \cdots + f_n$.
\end{lemma}

\section{Transforming univariate fractions into mixed fractions in $\Ratzd$}
In this section we consider the transformation of univariate fractions 
over the meadow of rational numbers to a mixed form.

For a finite set $A \subseteq \Ratzd$ we define the function 
$\varphi_{\!A} : \Ratzd \to \{0,1\}$ by
\begin{align*}
  \varphi_{\!A}(x) = 1 - \dfrac{\prod_{a\in A}(x-a)}{\prod_{a\in A} (x-a)} \,.
\end{align*}
Because $\varphi_{\!A}(a) = 1$ precisely when $a \in A$, 
$\varphi_{\!A}$ is equivalent to the characteristic function of $A$.

\begin{theorem}\label{thm:one:var}
  For every univariate term $t$ over the signature~$\sigdmd$,
  there is a polynomial~$g$ and a simple fraction $f$
  such that\, $\Ratzd \models t = g + f$.
\end{theorem}

\begin{proof}
  Let $t$ be a term as stated.
  From Lemma~\ref{lem:cees}
  we know that $t$ can be written 
  as a sum of simple fractions
  \begin{align*}
    t' \defd \dfrac{p_1}{q_1} + \dfrac{p_2}{q_2} + \cdots + \dfrac{p_n}{q_n}
  \end{align*}
  such that $\eqnsdmd \proves t = t'$ and so $\Ratzd \models t = t'$.
  Without loss of generality, we may assume each of the summands of $t'$ to be nonzero,
  that is, for every $i \in \fin{n}$, 
  there exist $u,v \in \Ratzd$ such that 
  $\Ratzd \models p_i(u) \ne 0 \wedge q_i(v) \ne 0$. 
  Hence, for every $i \in \fin{n}$, $p_i$ and $q_i$ 
  are non-trivial polynomials with integer coefficients (of the single variable~$x$).

  Now let $u$ denote the following term 
  that results from the ordinary way of turning $t'$ into a simple fraction:
  \begin{align*}
    u \defd \dfrac{\sum_{i = 1}^{n}(p_i \cdot \prod_{j = 1}^{i-1} q_j \cdot \prod_{j = i+1}^n q_j)}{\prod_{j = 1}^n q_j}
  \end{align*}
  Clearly, $u$ is a simple fraction, and so if $\Ratzd \models t = u $ we are done. This need not be the case, however,
  e.g., with $t = 1/x + 1/1$ one obtains $u = (1+x)/x$ and then $\Ratzd \not\models t = u $. Note that we may have 
  $\Ratzd \models t = u$ while $\Complzd \not\models t = u$, for instance with $t= 1/(x^2-2) + 1/1$ one obtains $u = (x^2-1)/(x^2-2)$.
  
    
Now assuming that $\Ratzd\not  \models t = u $ it must be the case that $q_i$ takes value 0
  for some~$i$. 
  More precisely, for every $a$ with $\Ratzd \not \models t(a) =u(a)$
  there exists $i \in \fin{n}$ 
  such that $\Ratzd \models q_i(a) = 0$. 

  Being a non-trivial univariate polynomial the equation $q_i = 0$ has only finitely many solutions.
  For $i \in \fin{n}$ we let $A_i$ denote the set of solutions of $q_i$,
  \begin{align*}
    A_i \defd \set{\,a \in \Ratzd \where \Ratzd \models q_i(a) = 0\,}\,. 
  \end{align*}
  Moreover, we define $A \defd \Union{i\in\fin{n}} A_i$, 
  and let $a_1,a_2,\ldots,a_k$ be an enumeration of $A$ without repetition.
  Note that 
  we have $t (a) \ne u(a)$ only if $a \in A$,
  and also that $a \in A$ implies $u(a) = 0$. Furthermore,
  for $i \in [k]$ we let 
  \begin{align*}
    b_i \defd 
    \dfrac{t(a_i)}{\prod_{j=1}^{i-1} \left(a_i - a_j\right) \cdot \prod_{j=i+1}^{k} \left(a_i - a_j\right)} \,.
  \end{align*}

  By minimality of $\Ratzd$ and \cite[Thm.~26]{BergstraM2016}
  we find for all $i \in \fin{k}$ 
  closed simple fractions $\cterm{a_i}$ and $\cterm{b_i}$.
   We will now define a polynomial~$g$ such that 
  \begin{align} 
    \Ratzd \models t  = u +  \varphi_{\!A} \cdot g \,.
    \label{eq:gg}
  \end{align}
  We define polynomials~$g_i$ for every $i \in \fin{k}$,
  and $g$ as follows:
  \begin{align*}
    g_i \defd \cterm{b_i} \cdot {\prod_{j=1}^{i-1} \left(x - \cterm{a_j}\right) \cdot \prod_{j=i+1}^{k} \left(x - \cterm{a_j}\right)} 
    &&
    g \defd \sum_{i=1}^k g_i \,.
  \end{align*}
  We verify that the identity~\eqref{eq:gg} holds.
  We have, in $\Ratzd$,
  \begin{align*}
    g_i(a_j) = 
    \begin{cases}
      0 & \text{if $i \ne j$,} \\
      t(a_i) & \text{if $i = j$,}
    \end{cases}
  \end{align*}
  and hence $g(a) = t(a)$ whenever $a \in A$.
  For $a\in \Ratzd$ we distinguish two cases:
  If $a \in A$ we find $ u(a) = 0$, 
  and $\varphi_{\!A}(a) \cdot g(a) = 1 \cdot t(a) = t(a)$.
  If $a \not\in A$ we see $t(a)=u(a) $ and also $\varphi_{\!A}(a) \cdot g(a) = 0 \cdot g(a) = 0$.
  
We will use the following abbreviations:
  \begin{align*}
    e \defd \prod_{i = 1}^k(x - \cterm{a_i}) \,,
    && q \defd \prod_{j = 1}^n q_j \,, 
    && u' = {\sum_{i = 1}^{n}\left(p_i \cdot \prod_{j = 1}^{i-1} q_j \cdot \prod_{j = i+1}^n q_j\right)} \,.
  \end{align*}
  It follows that $u = \tfrac{u'}{q}$.
  Now, working in~$\Ratzd$ we find:
  \begin{align*}
    t 
    & = u + \varphi_{\!A} \cdot g \\
    & = u + \left(1 - \dfrac{e}{e}\right) \cdot g \\
    & = \dfrac{u'}{q} - \dfrac{e\cdot g}{e}  + g \\
    & = \dfrac{u' \cdot e - q \cdot e \cdot g}{q \cdot e} + g\\
    &= \dfrac{u'  - q  \cdot g}{q } + g
  \end{align*}
  where the last two steps are justified by noticing that
  $q(x)$ and $e(x)$ both have the set~$A$ as their set of roots.

Since $g_i$ is a polynomial for every $i\in [k]$, the sum $g$ is a polynomial as well. In the appendix we show how $g$ can be rewritten to a standardised form
\[
c_{k-1}\cdot x^{k-1} + c_{k-2}\cdot x^{k-2} + \cdots + c_1\cdot x _1 + c_0
\]
where every $c_i$ is of the form $r_i/l$ with $r_i \in \zz$ and $l\in \Natpos$ (see also Example~\ref{example2}).
Using the cancellation axiom ($l\ne 0 \Rightarrow l/l=1$), it follows that $\Ratzd$ admits 
  the fraction~$(u'  - q \cdot g)/q$ to be rewritten into a simple fraction:
  \begin{align}
    \dfrac{u'  - q  \cdot g}{q }
    & = \dfrac{u' \cdot l - q \cdot  g \cdot l}{q \cdot l} \notag \\
    & = \dfrac{u' \cdot l- q   \cdot( c_{k-1}\cdot x^{k-1} + c_{k-2}\cdot x^{k-2} + \cdots + c_1\cdot x _1 + c_0) \cdot l }{q \cdot l}  \notag\\
    & = \dfrac{u' \cdot l- q \cdot (r_{k-1}\cdot x^{k-1} + r_{k-2}\cdot x^{k-2} + \cdots + r_1\cdot x _1 + r_0)}{q \cdot l} 
  \end{align}
\end{proof}

\begin{example}\label{example2}
We consider the case
\[
t= \dfrac{1}{x^2 +3x} + \dfrac{2x+5}{x^5 +1} + \dfrac{x^3 +2}{3x^2 - 7} .
\]
Then
\[
\begin{array}{rclcrclcrcl}
p_1& =& 1&\hspace{1cm}& p_2&=&2x +5& \hspace{1cm}&p_3&=&x^3 +2\\
q_1& =& x^2+3x&& q_2&=&x^5+1&& q_3&=&3x^2-7
\end{array}
\]
and hence
 \begin{align*}
u
    & = \dfrac{p_1q_2q_3 + p_2q_1q_3 + p_3q_1q_2}{q_1 q_2q_3} \\
    & = \dfrac{(x^5+1)(3x^2 - 7) + (2x+5)(x^2+3x)(3x^2-7) + (x^3 +2)(x^2 +3x)(x^5 +1)}{(x^2 +3x)(x^5+1)(3x^2 -7) }
   \end{align*}
Observe that for $x=0$ we have $t= 5 - 2/7$ whereas $u=0$.  The sets of roots are respectively $A_1=\{-3,0\}$, $A_2=\{-1\}$ en $A_3=\emptyset$. We choose the enumeration $a_1=-3$, $a_2=0$ en $a_3=-1$. It follows that 
\[
\begin{array}{rcl}
b_1& =& \dfrac{t(a_1)}{(a_1 - a_2)(a_1-a_3)}=\dfrac{1}{6}( \dfrac{1}{242} - \dfrac{5}{4})=-\dfrac{201}{968},\\
b_2& =& \dfrac{t(a_2)}{(a_2 - a_1)(a_2-a_3)}=\dfrac{1}{3}( 5 - \dfrac{2}{7})= \dfrac{11}{7},\\
b_3& =& \dfrac{t(a_3)}{(a_3 - a_1)(a_3-a_2)}=-\dfrac{1}{2}( -\dfrac{1}{2} - \dfrac{1}{4} )= \dfrac{3}{8}.\\
\end{array}
\]
We therefore obtain
  \begin{align*}
  g
    & = g_1 + g_2 + g_3 \\
    & = -\dfrac{201}{968}(x-0)(x+1) + \dfrac{11}{7}(x+3)(x+1) + \dfrac{3}{8}(x+3)(x-0)\\
    &=\dfrac{5891}{3388}x^2 +\dfrac{24404}{3388}x + \dfrac{15972}{3388}\
  \end{align*}
and the latter polynomial is in standardised form (see also the appendix). For this particular form we have $l=3388$ and $r_0=15972$, $r_1=24404$, $r_2= 5891$. Filling in the appropriate values in (2) yields the simple fraction
\[
\dfrac{u' - q \cdot g}{q}= \dfrac{u' \cdot 3388- (x^2 +3x)(x^5+1)(3x^2 -7) \cdot (5891\cdot x^{2} + 24404\cdot x  + 15972)}{(x^2 +3x)(x^5+1)(3x^2 -7) \cdot 3388}
\] 
where $u'=(x^5+1)(3x^2 - 7) + (2x+5)(x^2+3x)(3x^2-7) + (x^3 +2)(x^2 +3x)(x^5 +1)$.
 \end{example}

The transformation method outlined above is a computable process, since 
finding an enumeration of the different roots of the polynomials $q_i$ is a computable task. 

As a consequence of Theorem~\ref{thm:one:var} in combination with the previous remark 
we find that it is decidable whether or not a univariate fraction $t$ is equal to a simple fraction in $\Ratzd$. 
To see this first write $t$ as $f + g$ with $f$ a simple fraction (say $f = u/v$) and $g$ a polynomial. 
We may assume that $f$ and $g$ are both nonzero otherwise the issue is settled 
and a single simple fraction suffices to express $t$. 

Now consider the set $A$ of rational roots of $v$. 
We distinguish two cases:  if all values in $A$ are also roots of $g$ then 
\begin{align*}
  \Ratzd \models g =  \frac{g \cdot \prod_{a\in A}(x-a)}{\prod_{a\in A} (x-a)}  
\end{align*}
and we find that $t$ equals a simple fraction:
\begin{align*}
  t = \frac{u \cdot \prod_{a\in A} (x-a) + v \cdot g \cdot \prod_{a\in A} (x-a)}{v \cdot \prod_{a\in A} (x-a)} \,.  
\end{align*}
In the second case some $a \in A$ is not a root of $g$. 
Then $f$, viewed as a real function in $x,$ must be discontinuous in $a$. 
Because $g$ is continuous everywhere $t$ viewed as a real function is discontinuous in $a$, 
and at the same time it is non-zero in $a$. 
This combination of facts excludes that $t$ is equal to a simple fraction, 
say $f = u/v$ because $f$ is only discontinuous in zeroes of $v$ where the value of $f$ must be 0.

\section{Working in the complex numbers}
In this section we will briefly sketch a modified version of the result which works in the meadow of
complex numbers. 
This is a stronger result than Theorem~\ref{thm:one:var}. 
The value of having both proofs is that
the methods differ somewhat and that in some cases the method implicit in 
Theorem~\ref{thm:one:var} yields simpler expressions as a result. 
For instance, $1/(x^2+1)+1/(x^2+2) = (2 x^2 + 3)/( (x^2+1)\cdot(x^2+2))$ 
is the correct outcome for the method implicit in 
Theorem~\ref{thm:one:var} but not for Theorem~\ref{thm:one:varb} (see also Example~\ref{example3}).

\begin{lemma}\label{lemma}
Let $r(x)$ be an irreducible polynomial and let $R(x)$ be a $\sigdmd$ expression.
Then there is a polynomial $s(x)$ such that $\Complzd, \models r(x) = 0 \to s(x) = R(x)$. 
\end{lemma}

\begin{proof}
    The lemma is shown by induction on the structure of  $R$, 
    the only interesting case being $R(x) = u(x)/v(x)$ for polynomials 
    $u(x)$ and $v(x)$ with $v(x)$ nonzero. 
    We assume that $r(x)\neq v(x)$; otherwise we put $s(x)=0$.
   As $r(x)$ is irreducible the gcd of $r(x)$ and $v(x)$ equals $1$ and the Euclidean algorithm 
    provides polynomials $r'(x)$ and $v'(x)$ such that 
    $\Complzd  \models r(x) \cdot v'(x) + v(x) \cdot r'(x) = 1$.
    In this case we take $s(x) = u(x)\cdot r'(x)$. 
    Then if $\Complzd  \models r(a) = 0$ it must be the case that 
    $\Complzd  \models v(a) \cdot r'(a) = 1$ and thus $\Complzd  \models r'(a) = 1/v(a)$ 
    so that $\Complzd\models s(a) = u(a) /v(a) = R(a)$. Because this works for all $a$ in $\Complzd$ we have
    $\Complzd \models r(x) = 0 \to R(x) = s(x)$.
\end{proof}

\begin{theorem}\label{thm:one:varb}
  For every univariate term $t$ over the signature~$\sigdmd$,
  there are a polynomial $g$ and a simple fraction $f$
  such that $\Compl^d_0 \models {t = g + f}$.
\end{theorem}
   
\begin{proof}
  Let $t$ be a term over $\sigdmd$ with zero or more occurrences of a variable~$x$.
  Using Lemma~\ref{lem:cees}
  $t$ can be written as a sum of simple fractions,
  \begin{align*}
    \eqnsdmd  \vdash {t = \mfrac{p_1}{q_1} + \mfrac{p_2}{q_2} + \cdots + \mfrac{p_n}{q_n}} \,,
  \end{align*}
  where, without loss of generality, unless
  $\eqnsdmd  \vdash {t =0},$ each of the summands is assumed to be nonzero.
  Thus for every $i \in [n]$, there exists $a \in \Complzd$ such that 
  $p_i(a) \ne 0$ and $q_i(a) \ne 0$. 
  We may assume that all $p_i$ and $q_i$ ($i \in [n]$) are non-trivial polynomials 
  in~$x$ with integer coefficients.
  
  Focusing on the $q_i$ these can be uniquely factorised into products of irreducible polynomials 
  with integer coefficients.
  Therefore for $i \in [n]$ we may choose irreducible (over $\Rat_0$) and univariate polynomials 
  $q_{i,1},\ldots,q_{i,l_i}$ with integer coefficients such that 
  for appropriate natural numbers $f_{i,j}>0$,
  $\eqnsdmd  \vdash q_i = (q_{i,1})^{f_{i,1}} \cdots (q_{i,l_i})^{f_{i,l_i}}$. 
  Now we consider any enumeration without repetition $r_i, i \in [m]$ 
  of all irreducible polynomials just mentioned as factors of any of the $q_i$. 
  We write $e(i)$ for the sum over all $j \in [n]$ of the exponents 
  of the polynomial $r_i$ in the factorisation of $q_j$.
  It follows from this construction that $e(i) \geq 1$ for all $i \in [m]$.
  Following the proof of Theorem~\ref{thm:one:var}, 
  $u$ is introduced as follows:
  \begin{align*}
    u \defd \mfrac{\sum_{i = 1}^{n}(p_i \cdot \prod_{j = 1}^{i-1} q_j \cdot \prod_{j = i+1}^n q_j)}{\prod_{j = 1}^n q_j}.
  \end{align*}
  
  If $\Complzd\not \models t =u$ it must be the case that $ r_i$ takes value 0 for some $i\in [m]$.
  More precisely, for every $a$ with $\Complzd \not \models t (a) =u(a)$,
  there exists a unique $i \in [m]$  
  such that $\Complzd\models  r_i(a) = 0$. 
  Unicity of $i$ follows from the fact that different irreducible and univariate polynomials 
  cannot share roots. 

  We define the expression $\varphi$ as follows:
  \begin{align*}
    \varphi(x) = 1 - \mfrac{\prod_{i=1}^{m} (r_i(x))^{e(i)}}{ \prod_{i=1}^{m} (r_i(x))^{e(i)}} 
    \,.
  \end{align*}
We will now define a polynomial~$g$
  such that 
  \begin{align}
    \Complzd \models t = u + \varphi  \cdot g 
  \end{align}


  For each $i \in [m]$ we introduce the polynomial $h_i(x)$ as follows:
  \begin{align*}
    h_i(x) \defd \prod_{j=1}^{i-1} r_j(x) \cdot \prod_{j=i+1}^{m} r_j(x) \,.
  \end{align*}
  Using Lemma~\ref{lemma} we may choose for each $i \in [m]$
  a polynomial $s_i(x)$ such that 
  \begin{align*}
    \Complzd \models r_i(x)=0 \to s_i(x) = \frac{t(x)}{h_i(x)} \,.
  \end{align*} 
  Further for $i \in [m]$ the polynomial 
  $g_i$ is defined by
  \begin{align*}
    g_i(x) \defd h_i(x) \cdot s_i(x) \,.   
  \end{align*}
  We notice that if $r_i(a) = 0$, then $h_i(a) \ne 0$. 
  Otherwise for some $j$ different from $i$ the different univariate polynomials $r_i(a)$ and $r_j(a)$ 
  are both zero, which is impossible as these are both irreducible. 
  Now we define $g$ by 
  \begin{align*}
    g(x) \defd \sum_{i=1}^k g_i(x) \,.
  \end{align*}
  (3) can be verified as follows. Either $r_i(a)\neq 0$ for all $i\in [m]$; in that case we have $t=u$ and $\varphi \cdot g = 0$. Otherwise there is a single $i$ with
 $r_i(a) = 0$ (in $\Complzd$) and hence
   $g_i(a) =  t (a) $;
  for $j \neq i$ ($j \in [m]$) we have $h_j(a) = 0$  and therefore $g_j(a) = 0$. 
  As $g_i$ is a polynomial for every $i \in [k]$ the sum $g$ is a polynomial as well.
  Thus we obtain: 
  \begin{align*}
    \Complzd \models t=  u + \varphi \cdot g 
    = u + (1 -
      \mfrac{%
        \prod_{j=1}^{m} r_j(x)
      }{%
        \prod_{j=1}^{m} r_j(x)
      }) \cdot g 
    = u + g - 
      \mfrac{%
        \prod_{j=1}^{m} r_j(x) 
      }{%
         \prod_{j=1}^{m} r_j(x)
      } \cdot g
      \,.
  \end{align*}
  Just as it was done in the proof of Theorem~\ref{thm:one:var} this expression can be rewritten 
  into a sum of a polynomial and a simple fraction. 
  Rewriting the resulting expression in such a manner that only closed subexpressions for integers 
  occur in numerators and denominators can be done in a similar manner as well. 
\end{proof}
\begin{example}\label{example3}
We return to the example mentioned in the beginning of this section. If $t=1/(x^2+1)+1/(x^2+2)$ then
\[
u= \frac{2x^2 +3}{(x^2+1)(x^2 +2)}
\]
and we find that $s_1(x)=1=s_2(x)$. It follows that $g_1(x)=h_1(x)=x^2 +2$, $g_2(x)= x^2 +1$ and hence $g(x)=2x^2 + 3$. Thus
\[
t(x)=\frac{2x^2 +3}{(x^2+1)(x^2 +2)} + (1 - \frac{(x^2+1)(x^2 +2)}{(x^2+1)(x^2 +2)}) \cdot (2x^2 +3)
\]
and the right expression can be rewritten to the mixed fraction
\[
\frac{(2x^2 +3)(1 - (x^2+1)(x^2 +2))}{(x^2+1)(x^2 +2)} + (2x^2+3).
\]
\end{example}

Using the fact that satisfaction in $\Compl^d_0$ coincides with provability from 
$\eqnsdmd \cup\ \{ \underline{n} \mdiv \underline{n} = 1 \where n > 0 \}$ 
(see~\cite{Ono1983,BB2016}),
we find that in this case also: $\eqnsdmd \cup \{\underline{n}/\underline{n} = 1 \where n>0\} \vdash t = g + f. $ 
Two assumptions  $\underline{k}/\underline{k}=1$ and $\underline{l}/\underline{l}=1$ can be encoded in a single one:
$\underline{k \cdot l}/\underline{k \cdot l}=1$, which leads to the following fact.

\begin{theorem}\label{thm:one:varc}
  For every univariate term $t$ over the signature~$\sigdmd$,
  there are a polynomial~$g$, a simple fraction~$f$, and a natural number~$n>0$
  such that 
  $\eqnsdmd \cup\, \{\underline{n}/\underline{n}=1\} \vdash t = g + f. $ 
  (Or equivalently: $\eqnsdmd \vdash \underline{n} \cdot t = \underline{n} \cdot (g + f). $)
 \end{theorem}
 
 In \cite{Bergstra2016} a \emph{finite support summation}  $\sum^*_x t(x)$ is introduced. This expression denotes in a meadow the sum of all $t(a)$ if only finitely many of these substitutions yield a nonzero value, and 0 otherwise. So, if $t(x)= 1 - \frac{x}{x}$ then $\sum^*_x t(x) = 1$ in $\Compl^d_0$, since $t(x)\neq 0$ only if $x=0$; however, if $t(x)$ is a nonzero complex polynomial then $\sum^*_x t(x)= 0$ since $t(x)$ has only finitely many roots in the complex numbers. In the following theorem we apply Theorem~\ref{thm:one:varb}.
\begin{theorem}\label{thm:two:varb}
  For every univariate term $t$ over the signature~$\sigdmd$,
 it is decidable whether or not 
  $\Compl^d_0 \models {\sum^*_x t(x)=1}$.
\end{theorem}
\begin{proof}
Using Theorem~\ref{thm:one:varb} $t$ can be written as  the sum of a polynomial and a simple fraction $g+ \frac{p}{q}$.  We distinguish two cases.
\begin{enumerate}
\item $\Compl^d_0 \models g+ \frac{p}{q}\neq 0$: Then there exists a complex number $a$  and an $\epsilon$-neighbourhood of $a$ on which $t$ interpreted as a complex, continous  function must be nonzero. Hence $\Compl^d_0 \models {\sum^*_x t(x)=0}$.
\item $\Compl^d_0 \models g+ \frac{p}{q} = 0$: If $\Compl^d_0 \models q = 0$, then $\Compl^d_0 \models g= 0$ and hence $\Compl^d_0 \models {\sum^*_x t(x)=0}$. So assume $\Compl^d_0 \models q \neq 0$. Then $\Compl^d_0 \models p= -g\cdot q$. It follows that 
\[
\Compl^d_0 \models t=g+ \frac{p}{q}=g- \frac{g\cdot q}{q}= (1 - \frac{q}{q})\cdot g.
\]
Thus $t(a)\neq 0$ if and only if $a$ is a root of $q$ and $g(a)\neq 0$. From the degree of $q$ we can effectively derive the number $m$ of roots of $q$. Now observe that
\[
\begin{array}{rcll}
\sum^*_x t(x)=1 & \Leftrightarrow & \exists x_1, \ldots , x_m (& 
q(x_1) = 0 \ \wedge\  \cdots \ \wedge\ q(x_m)=0 \ \wedge \\[3mm]
&&& \bigwedge_{1 \leq i < j \leq m} \frac{x_i - x_j}{x_i -x_j} =1 \ \wedge \\[3mm]
&&&\forall y ( q(y) = 0 \rightarrow y = x_1 \ \vee \cdots \ \vee y = x_m) \ \wedge \\[3mm]
&&& g(x_1) + \cdots + g(x_m)=1)
\end{array}
\]
Decidability now follows from Tarski's decidability result of the first order theory of real closed fields.
\end{enumerate}
\end{proof}
Observe that the term 1 in the above theorem can be replaced by an arbitrary closed term.

\section{Concluding remarks} 
These results are quite specific for the case of involutive meadows, that is meadows satisfying $(x^{-1})^{-1}= x$. 
In common meadows \cite{BergstraP2015} the issues disappear because common meadows unconditionally 
satisfy $x/y + u/v = (x \cdot v + y \cdot u) /(y \cdot v)$.
In wheels \cite{Carlstroem2004}, and also in transreals \cite{ReisA2014} the issues also trivialise but in a 
somewhat different manner.  

Many related questions await further exploration. 
In particular the question to what extent these results can be generalised 
to the case of multivariate expressions is intriguing. 
We were unable to make any progress in the case of expressions with two variables, 
nevertheless it seems plausible to expect that expressions with two variables 
(that is, sums of simple fractions involving two variables) 
can all be transformed to a uniformly bounded sum of simple fractions.


\newpage

\appendix
\section{Writing polynomials to standard form}
  In the proof of Theorem~\ref{thm:one:var} the term $g$ was defined by
  \begin{align*}
    g \defd \sum_{i=1}^k g_i
    &&
    g_i \defd \cterm{b_i} \cdot {\prod_{j=1}^{i-1} \left(x - \cterm{a_j}\right) \cdot \prod_{j=i+1}^{k} \left(x - \cterm{a_j}\right)} 
    && (i \in \fin{k}) \,.
  \end{align*}
  Here we show how $g$ can be rewritten into a standardised polynomial with rational coefficients.
  Given a finite set $S$ and an integer $k \ge 0$, 
  we use $\binom{S}{k}$ to denote the set of subsets of cardinality~$k$, that is,
  $\binom{S}{k} = \set{I \subseteq S \where \card{I} = k}$.
  Clearly, $\binom{S}{k}$ has cardinality $\binom{\card{S}}{k}$.
  For example, $\binom{\fin{4}}{3} = \set{\set{1,2,3},\set{1,2,4},\set{1,3,4},\set{2,3,4}}$.
  
  Let $i \in \fin{k}$, and $H_i = \fin{k} \setminus \set{i}$.
  In $\Ratzd$ we have 
  \begin{align*}
    g_i 
    & = b_i \cdot \prod_{j \in H_i} \left(x - a_j\right) \\ 
    & = c_{i,0} \cdot x^{k-1} + c_{i,1} \cdot x^{k-2} + \cdots + c_{i,k-2} \cdot x + c_{i,k-1} \,,
  \end{align*}
  where, for $0 \le j \lt k$,
  \begin{align*}
    c_{i,j} = (-1)^j \cdot b_i \cdot \sum_{I \in \binom{H_i}{j}} \prod_{h \in I} a_h \,.
  \end{align*}
  %
 For $i\in [k]$ we let $n_i, r_i \in \mathbb{Z}$ and $m_i, s_i \in \mathbb{N}^+$ be such that 
 $b_i = \tfrac{r_i}{s_i}$ and $a_i = \tfrac{n_i}{m_i}$
  and turn the above expression for $c_{i,j}$ into one fraction 
  using that $\tfrac{s_h}{s_h} = 1$ and $\tfrac{m_h}{m_h} = 1$ for all $h \in \fin{k}$.
  We thus obtain
  \begin{align*}
    c_{i,j} = \dfrac{(-1)^j \cdot r_i \cdot \sum_{I \in \binom{H_i}{j}} \prod_{h \in I} n_h \cdot \prod_{h' \in S \setminus I} m_{h'}}{s_i \cdot \prod_{h \in H_i} m_h} \,.
  \end{align*}
  To get ready for summing up the $c_{i,j}$ (for fixed $j$ and varying $i$),
  we transform $c_{i,j}$ into one denominator common for all $i \in \fin{k}$:
  \begin{align*}
    c_{i,j} = \dfrac{(-1)^j \cdot r_i \cdot m_i \cdot \prod_{\ell \in H_i} s_\ell \cdot \sum_{I \in \binom{H_i}{j}} \prod_{h \in I} n_h \cdot \prod_{h' \in S \setminus I} m_{h'}}{s \cdot m} \,,
  \end{align*}
  where $s = \prod_{j \in \fin{k}} s_j$ and $m = \prod_{j \in \fin{k}} m_j$.
  
  Finally, summing up the polynomials $g_i$, we get 
  \begin{align*}
    g = \sum_{i \in \fin{k}} g_i 
      = c_{0} \cdot x^{k-1} + c_{1} \cdot x^{k-2} + \cdots + c_{k-2} \cdot x + c_{k-1} \,,
  \end{align*}
  where 
  \begin{align*}
    c_j & = \dfrac{(-1)^j \cdot \sum_{i\in\fin{k}} \left( r_i \cdot m_i \cdot \prod_{\ell \in H_i} s_\ell \cdot \sum_{I \in \binom{H_i}{j}} \prod_{h \in I} n_h \cdot \prod_{h' \in S \setminus I} m_{h'} \right)}{s \cdot m} \,. 
  \end{align*}

\end{document}